\documentclass{amsart}
\usepackage{amsmath}
\usepackage{amssymb}
\usepackage{amsthm,amssymb,amsmath,graphicx}
\usepackage{mathtools}
\usepackage{pgf,tikz}
\usepackage{mathrsfs}
\usepackage{enumitem}

\usetikzlibrary{arrows}

\makeatletter
\newtheorem*{rep@theorem}{\rep@title}
\newcommand{\newreptheorem}[2]{%
\newenvironment{rep#1}[1]{%
 \def\rep@title{#2 \ref{##1}}%
 \begin{rep@theorem}}%
 {\end{rep@theorem}}}
\makeatother

\newtheorem{theorem}{Theorem}[section]
\newreptheorem{theorem}{Theorem}

\newtheorem{lemma}[theorem]{Lemma}

\newtheorem{proposition}[theorem]{Proposition}

\newtheorem{observation}[theorem]{Observation}
\newtheorem{corollary}[theorem]{Corollary}

\newtheorem{conjecture}[theorem]{Conjecture}

\theoremstyle{remark}

\newtheorem{remark}[theorem]{Remark}

\newtheorem{example}[theorem]{Example}

\begin{document}

\makeatletter

\makeatother
\author{Ron Aharoni}
\address{Department of Mathematics\\ Technion, Haifa, Israel}
\thanks{The research of the first author was  supported by  an ISF grant, BSF grant no. 2006099 and by the Discount Bank Chair at the Technion.}
\email[Ron Aharoni]{raharoni@gmail.com}

\author{Matthew DeVos}
\address{Department of Mathematics \\ Simon Fraser University, Burnaby, B.C., Canada}
\email[Matthew DeVos]{mdevos@sfu.ca}

\author{Ron Holzman}
\address{Department of Mathematics \\ Technion, Haifa, Israel }

\email[Ron Holzman]{holzman@technion.ac.il}

\title{Rainbow triangles and the Caccetta-H\"aggkvist conjecture}

\begin{abstract}
A famous conjecture of Caccetta and H\"aggkvist is that in a digraph on $n$ vertices and minimum out-degree at least $\frac{n}{r}$ there is a directed cycle of length $r$ or less. We consider the following generalization: in an undirected graph on $n$ vertices, any collection of $n$ disjoint sets of edges, each of size at least $\frac{n}{r}$, has  a rainbow cycle of length $r$ or less. We focus on the case $r=3$, and prove the existence of a rainbow triangle under somewhat stronger conditions than in the conjecture. For any fixed $k$ and large enough $n$, we determine the maximum number of edges in an $n$-vertex edge-coloured graph where all colour classes have size at most $k$ and there is no rainbow triangle. Moreover, we characterize the extremal graphs for this problem.

\end{abstract}

\maketitle

\section{Introduction}


The following conjecture is one of the best known in graph theory:

\begin{conjecture}[Caccetta-H\"aggkvist \cite{ch}]
For every positive integer $r$, every digraph on $n$ vertices with minimum out-degree at least $\frac{n}{r}$ has a directed cycle of length at most~$r$.
\end{conjecture}

This conjecture is trivial for $r \le 2$ and has been proved for $r \ge \sqrt{2n}$ by Shen \cite{s00}.  The special case  $r=3$ has received considerable attention, see e.g., \cite{bon,hhk,r13,s98}; the best known bound in this case is given by a theorem of Hladk\'y, Kr\'al' and Norin \cite{hkn} asserting that every simple digraph on $n$ vertices with minimum out-degree at least $0.3465n$ contains a directed triangle.

The following conjecture of the first author (unpublished),  is a  generalization of this conjecture, which concerns undirected, simple graphs. It replaces the notion of ``directed cycle'' by that of ``rainbow cycle''.  Define an \emph{edge-coloured graph} to be a graph $G = (V,E)$ equipped with a distinguished partition $\{C_1, C_2, \ldots, C_p\}$ of $E$, called a \emph{colour partition}. The sets $C_i$ are called \emph{colour classes},  the indices $i \in \{1,\ldots,p\}$ are called \emph{colours}, and we say that an edge $e$ has colour $i$ if $e \in C_i$.  A subgraph $H$ of $G$ is called \emph{rainbow} if distinct edges of $H$ have distinct colours.

\begin{conjecture}
\label{aharoniconj}
Let $r$ be a positive integer and let $G$ be an edge-coloured graph on $n$ vertices.  If there are $n$ colour classes, all of size at least $\frac{n}{r}$, then $G$ contains a rainbow cycle of length at most $r$.
\end{conjecture}

To see that Conjecture \ref{aharoniconj} implies the original, let $r$ be a positive integer and let $D$ be a directed simple graph on a set $V$ of $n$ vertices with minimum outdegree at least $\frac{n}{r}$.  Let $G$ be the undirected underlying graph of $D$ (which has the same vertex and edge set).  Extend $G$ to an edge-coloured graph with the colour partition $\{ \delta^+_D(v) \mid v \in V \}$ (here $\delta^+_D(v)$ denotes the set of edges of $D$ incident with $v$ and  directed away from it).  Conjecture \ref{aharoniconj} implies that $G$ has a rainbow cycle of length at most $r$, and it is easy to see that it corresponds to a directed cycle of length at most $r$ in the original digraph $D$.

Since the Caccetta-H\"aggkvist conjecture is sharp, if true, so is this conjecture. The standard example showing
that the requirement on the size of the colour classes cannot be relaxed is obtained by taking a cycle of length $n=kr+1$ with vertices $v_1,\ldots ,v_{n}$, and defining  $C_i$ ($1 \le i \le n$) to be the  set of all edges $v_iv_{i+j},~ 1 \le j\le k$ (indices taken cyclically).
Conjecture \ref{aharoniconj} is sharp also in another sense, namely  the requirement on the number of colours cannot be relaxed.
An example showing that $n-1$ colours do not suffice is
obtained by adding a new vertex $z$, which is not part of the cycle, to the previous example (so now $n=kr+2$), and adding to each $C_i$ ($1 \le i \le n-1$) the edge $v_iz$. (Below, in Example \ref{example35}, the case $r=3$ of this example will serve yet another purpose.)

In this paper we focus on the case $r=3$ of Conjecture \ref{aharoniconj}. This leads us to seek conditions on an edge-coloured graph that guarantee the existence of a rainbow triangle. Such conditions have been studied in the literature, e.g., \cite{ess,ga,gsa,gsi,lnxz,l,lw}.  The feature special to the present paper is that the size of the colour classes plays the key role.

\section{Two observations}
Some straightforward arguments give reasonably good approximations to Conjecture \ref{aharoniconj} for the special case  $r=3$.  Henceforth we use the term \emph{triangle} for a subgraph isomorphic to $K_3$.  The  case  $r=3$ of Conjecture \ref{aharoniconj} says that every edge-coloured graph on $n$ vertices with $n$ colour classes, each of size at least $\frac{n}{3}$, contains a rainbow triangle.  The following theorem gives two approximations to this.

\begin{theorem}
\label{chapprox}
Let $G$ be an edge-coloured graph on $n$ vertices.  If either of the conditions below is satisfied, then $G$ has a rainbow
triangle.
\begin{enumerate}
\item There are at least $\frac{9}{8}n$ colour classes, each of size at least $\frac{n}{3}$.
\item There are at least $n$ colour classes, each of size at least $\frac{2n}{5}$.
\end{enumerate}
\end{theorem}

To prove these statements, for every graph $G$ let $t(G)$ denote the number of triangles in $G$.

\begin{lemma}
\label{triangobs}
Let $G$ be an edge-coloured graph with $m$ edges without rainbow triangles.  If every colour class has size at most $k$, then $t(G) \le \frac{1}{2}m(k-1)$.
\end{lemma}

\begin{proof}

Let $\{C_1, \ldots, C_p\}$ be the colour partition of $G$.  For a triangle, there is a colour class containing the majority of its edges.  Every colour $i$ can be the majority for at most ${|C_i| \choose 2}$ triangles.  Thus
\[ t(G)
	\le  \sum_{i=1}^p {|C_i| \choose 2} 	
	= \tfrac{1}{2} \sum_{i=1}^{p} |C_i|^2 - \tfrac{1}{2}m	
	 \le \tfrac{1}{2} (\tfrac{m}{k} k^2) - \tfrac{1}{2}m	
	= \tfrac{1}{2}m(k-1). \]\end{proof}

The lemma bounds from above the number of triangles which can appear in an edge-coloured graph without rainbow triangles.  On the other hand, there is a rather straightforward lower bound on the number of triangles in a graph with $n$ vertices and $m$ edges, which follows from a nice application of the Cauchy-Schwarz inequality.

\begin{theorem}[Goodman \cite{go}, see also Nordhaus-Stewart \cite{ns}]
\label{goodman}
If $G$ is a graph with $n$ vertices and $m$ edges,
\[ t(G) \ge \frac{4m}{3n} \left( m - \frac{n^2}{4} \right). \]
\end{theorem}

Note that this implies Mantel's Theorem \cite{m} -- the special case of Tur\'an's Theorem which asserts that an $n$-vertex graph with more than $\frac{n^2}{4}$ edges contains a triangle.  Goodman's Theorem was improved by Bollob\'as \cite{bol}, and recently Razborov \cite{r08} proved a theorem establishing the precise dependence between the edge density and triangle density.  (This result was then extended by Nikiforov \cite{n} to $K_4$ subgraphs and by Reiher \cite{r} for cliques of arbitrary size.)  However, for the purpose of our application, using Razborov's theorem instead of Goodman's provides no improvement.

\bigskip

\noindent{\it Proof of Theorem \ref{chapprox}: } Assume (for a contradiction) that $G$ is an edge-minimal graph which violates one of the conditions.  Then the size of each colour class is $k$, where $k = \lceil \frac{n}{3} \rceil$ in the first case and  $k = \lceil \frac{2n}{5} \rceil$ in the second.  Now, combining Lemma \ref{triangobs} and Theorem \ref{goodman} we obtain the inequality $3n(k-1) \ge 8m - 2n^2$. As $m \ge \frac{9}{8}nk$ in the first case and $m \ge nk$ in the second, we get a contradiction in both cases.
\quad\quad$\Box$

\bigskip
\section{Bounded-size colour classes }

For positive integers $n,k$ define $g(n,k)$ to be the largest integer $m$ so that there exists an edge-coloured graph with $n$ vertices and $m$ edges, so that all colour classes have size at most $k$ and there is no rainbow triangle.  If we knew that $g(n, \lceil \frac{n}{3} \rceil) < n \lceil \frac{n}{3} \rceil$, then the $r=3$ case of Conjecture \ref{aharoniconj} would follow. Thus, determining $g(n,k)$ for all $n$ and $k$ appears at least as difficult as  the $r=3$ case of the Cacceta-H\"aggkvist conjecture. But it is possible to calculate $g(n,k)$ for $n$ suitably larger than $k$. This is the main content of this paper.

We start by introducing a family of edge-coloured graphs that, as we shall later prove, achieve the optimal bound for $g(n,k)$ whenever $n$ is sufficiently large.
To define this family, let $k,a,b$ be positive integers. Consider a complete bipartite graph $K_{a,b}$ with bipartition $(A,B)$ where $|A| = a$ and $|B| = b$.  Form a new graph $H_{a,b}^k$ from this graph by adding $\lfloor \frac{b}{k} \rfloor$ disjoint copies of $K_{k}$ on the set $B$, and then if $r = b - k \lfloor \frac{b}{k} \rfloor > 0$ we add another clique of size $r$ to $B$, disjoint from the rest.  Next we equip this graph with an edge colouring.  Let $B_1, \ldots, B_{ \lceil \frac{b}{k} \rceil}$ be the cliques which were added to $B$, and for every $v \in A$ and $1 \le i \le \lceil \frac{b}{k} \rceil$ declare the set of edges between $v$ and $B_i$ to be a colour class. Then, for each $i$,  order the vertices of  $B_i$ as $u_1, \ldots, u_{k'}$ (where $k'=|B_i|$),  and for every $1 \le j \le k'-1$ declare the set of edges between $u_j$ and $\{u_{j+1}, \ldots, u_{k'}\}$ to be a colour class.  It is straightforward to verify that all colour classes have size at most $k$, and there is no rainbow triangle.

Remark: the colour classes which contain edges with both ends in $B$ have sizes ranging from $1$ up to $k-1$.  If one wishes to obtain a colouring where all colour classes have size equal to $k$, this can easily be arranged in the case when $2k$ divides $b$ by merging colour classes from this construction.

When $k=1$ all subgraphs are rainbow, so the Tur\'an graph $K_{ \lfloor \frac{n}{2} \rfloor, \lceil \frac{n}{2} \rceil}$ (which is
also isomorphic to $H^1_{ \lfloor \frac{n}{2} \rfloor, \lceil \frac{n}{2} \rceil }$) has the maximum number of edges without a rainbow triangle.  Thus
\[ g(n,1) =  \left\lfloor \frac{n}{2} \right\rfloor \left\lceil \frac{n}{2} \right\rceil. \]

Next consider the case  $k=2$.

\begin{theorem}\label{gn2}
The values of $g(n,2)$ are given by:
\begin{enumerate}
\item $g(5,2)=8$,
\item $g(n,2)=\left\lfloor \frac{n}{2} \right\rfloor \left\lceil \frac{n}{2} \right\rceil  + \left\lfloor \frac{1}{2} \left\lceil \frac{n}{2} \right\rceil \right\rfloor$ for $n \ne 5$.

\end{enumerate}
\end{theorem}

\begin{proof}

To see that
 $g(5,2) \ge 8$ let $G$ be the $5$-wheel with center $z$ and cycle vertices $v_1, v_2, v_3, v_4$, and let the colour classes be $C_i=\{v_iv_{i+1},v_iz\} ~~(1 \le i \le 4)$, where $v_5:=v_1$. To see that $g(n,2)\ge  \left\lfloor \frac{n}{2} \right\rfloor \left\lceil \frac{n}{2} \right\rceil + \left\lfloor \frac{1}{2} \left\lceil \frac{n}{2} \right\rceil \right\rfloor$, notice that
 the graph $H^2_{ \lfloor \frac{n}{2} \rfloor, \lceil \frac{n}{2} \rceil }$ has $n$ vertices and $ \left\lfloor \frac{n}{2} \right\rfloor \left\lceil \frac{n}{2} \right\rceil + \left\lfloor \frac{1}{2} \left\lceil \frac{n}{2} \right\rceil \right\rfloor$ edges,  setting this as a lower bound on $g(n,2)$.

 For an upper bound, we shall use the following stability version of Tur\'an's Theorem.

\begin{theorem}[Lov\'asz-Simonovits \cite{ls}]\label{lovaszsimonovitz}
If $G$ is a graph with $n$ vertices and $\lfloor \frac{n}{2} \rfloor \lceil \frac{n}{2} \rceil + s$ edges, where $s < \lfloor \frac{n}{2} \rfloor$,  then $t(G) \ge s \lfloor \frac{n}{2} \rfloor$.
\end{theorem}

The upper bound can easily be checked for $n \le 5$. We henceforth assume that $G$ is a graph with $n \ge 6$ vertices and $m$ edges, with edges coloured so that the colour classes are of size at most $2$, and $G$ has no rainbow triangles.
Let $s=m - \lfloor \frac{n}{2} \rfloor \lceil \frac{n}{2} \rceil$.
  We have to show that $s \le \frac{1}{2} \lceil \frac{n}{2} \rceil$.

  Assume first that $s <  \lfloor \frac{n}{2} \rfloor$. In this case, combining Lemma \ref{triangobs} (for  $k=2$) with  Theorem \ref{lovaszsimonovitz} gives
  $$s \lfloor \frac{n}{2}\rfloor \le \frac{m}{2}.$$

\noindent Replacing $m$ by $\lfloor \frac{n}{2} \rfloor \lceil \frac{n}{2} \rceil +s$ and multiplying by $\frac{2}{\lfloor \frac{n}{2} \rfloor}$ yields
 $$2s \le \lceil \frac{n}{2} \rceil + \frac{s}{\lfloor \frac{n}{2} \rfloor}.$$

\noindent Since $\frac{s}{\lfloor \frac{n}{2} \rfloor}$ is less than $1$, it may be removed from the right-hand side, and we obtain $s \le \frac{1}{2} \lceil \frac{n}{2} \rceil$ as desired.

 If $s\ge \lfloor \frac{n}{2} \rfloor$, we can remove edges from the graph until the number of edges is $\lfloor \frac{n}{2} \rfloor \lceil \frac{n}{2} \rceil + \lfloor \frac{n}{2} \rfloor -1$. By the above argument, the number of edges in the new graph does not exceed $\lfloor \frac{n}{2} \rfloor \lceil \frac{n}{2} \rceil + \frac{1}{2}\lceil \frac{n}{2} \rceil$. For all values of $n \ge 6$ except $n=7$, this is a contradiction because $\lfloor \frac{n}{2} \rfloor -1 > \frac{1}{2} \lceil \frac{n}{2} \rceil$. It remains to rule out the case when $G$ has $7$ vertices and $15$ edges. In this case, the average degree in $G$ is less than $5$. By removing a vertex of degree at most $4$, we are left with a graph on $6$ vertices having at least $11$ edges, contradicting the upper bound for $n=6$.\end{proof}

For $k>2$ we do not know the value of $g(n,k)$ for general $n$. But the following theorem gives the answer for large enough $n$.

\begin{theorem}
\label{main}
Let $n_1(k)=6k^5(k+1)^2$ and $n_2(k)=(n_1(k))^2$, and assume that $n \ge n_2(k)$. Then there exists $a$ so that  $H^k_{a, n-a}$ has the maximum number of edges over all $n$-vertex edge-coloured graphs for which all colour classes have size at most $k$ and there is no rainbow triangle. Moreover, any graph attaining this maximum is isomorphic to $H^k_{a,n-a}$ for some $a$.
\end{theorem}

This theorem is proved in the next section. Roughly speaking, it asserts that $$g(n,k) \approx \frac{n^2}{4} + \frac{(k-1)n}{4}$$ when $n$ is large.  In order to determine the exact value of $g(n,k)$ for all large enough $n$, we need to find the optimal choice of $a$ within the family of graphs $H^k_{a,n-a}$. This is a routine maximization problem, whose solution is stated without proof in the following proposition. It turns out that the answer depends on the remainder  $n$ leaves upon division by $2k$.

\begin{proposition}\label{best}
Let $n=t(2k) +r$, where $1 \le r \le 2k$.

\begin{enumerate}
\item If $1\le r \le k$ then the maximum of $|E(H^k_{a,b})|$ under the constraint $a+b=n$ is attained (solely) in the following two cases: $a=tk,~b=tk+r$, and $a=tk+1,~b=tk+r-1$. The value of the maximum is $\frac{n^2}{4} + \frac{(k-1)n}{4} - \frac{r(k+1-r)}{4}$.

\item If $k+1\le r \le 2k$ then the maximum of $|E(H^k_{a,b})|$ under the constraint $a+b=n$ is attained (solely) when  $a=(t-1)k+r,~b=(t+1)k$. The value of the maximum is $\frac{n^2}{4} + \frac{(k-1)n}{4} + \frac{(r-k-1)(2k-r)}{4}$.
\end{enumerate}
\end{proposition}

Theorem \ref{main} shows that if $n$ is large enough with respect to $k$, then $g(n,k)$
is attained at some graph $H^k_{a,b}$. This is not  true for all values of $n$. We have already noticed it in the case $n=5,~k=2$ (see Theorem~\ref{gn2}), where an example was given demonstrating $g(5,2) \ge 8$, while the optimal value of $|E(H^2_{a,b})|$ is $7$. This example can be generalized, as follows.

\begin{example}\label{example35}
Let $k \ge 2$ and $n=3k-1$. Take a cycle $v_1v_2\ldots v_{3k-2}$, add a special vertex $z$, and connect every $v_i$ to all vertices $v_j$ at distance at most $k-1$ from it on the cycle, and to $z$. The $3k-2$ colour classes are the sets $C_i = \{v_iv_{i+j} \mid 1 \le j \le k-1\} \cup \{v_iz\}$ for all $1 \le i \le 3k-2$, where counting is modulo $3k-2$.
\end{example}

It is easy to check that there are no rainbow triangles. The number of edges is $(3k-2)k=3k^2-2k$, while by Proposition \ref{best}
(applied with $t=1,~r=k-1$)
the optimal value of $|E(H^k_{a,b})|$ is
$3k^2-3k+1$.

\section{Proof of Theorem \ref{main}}

Here is a pleasing fact about complete edge-coloured graphs:

\begin{lemma}\label{completecoloured}
\label{nokn}
Every edge-coloured $K_n$ without rainbow triangles has a colour class of size at least $n-1$.
\end{lemma}

The lemma is an easy corollary of the following pretty theorem of Gallai.

\begin{theorem}[Gallai \cite{ga}, see Gy\'arf\'as-Simonyi \cite{gsi}]
If $G$ is an edge-coloured complete graph on at least two vertices without a rainbow  triangle, there is a nontrivial partition ${\mathcal P}$ of $V(G)$ satisfying:
\begin{enumerate}
\item If $P,Q \in {\mathcal P}$ satisfy $P \neq Q$, then all edges with one end in $P$ and the other in $Q$ have the same colour.
\item The set of edges with ends in distinct blocks of ${\mathcal P}$ has at most two colours.
\end{enumerate}
\end{theorem}

\noindent{\it Proof of Lemma \ref{completecoloured}:} We may assume $n \ge 2$ and apply the above theorem to choose a partition ${\mathcal P}$. If this is a bipartition, then the edge-cut it defines is monochromatic and has size at least $n-1$. If $|{\mathcal P}| = 3$, then of the three cuts it defines two must be of the same colour, meaning that
the graph has a nontrivial monochromatic edge-cut (which must have size $\ge n-1$).  Otherwise $|{\mathcal P}| \ge 4$ and the number of edges between blocks of ${\mathcal P}$ must be at least $2(n-1)$ which again forces a colour class to have size $\ge n-1$ (to see this last claim, note that the number of edges between blocks is one half of the sum over all blocks $P \in {\mathcal P}$ of the number of edges with exactly one end in $P$).
\quad\quad$\Box$

\bigskip
For a graph $G$, let $\delta(G)$ denote the minimum degree in $G$. If $X$ is a set of vertices of $G$,  let $G[X]$ denote the graph induced on $X$,  let $E(X)$ denote the set of edges of $G$ with both ends in $X$, and  let $e(X) = |E(X)|$. Also, for $x \in X$ let $d_X(x)$ denote the degree of $x$ in $G[X]$. If $Y$ is a
set of vertices disjoint from $X$, then  let $E(X,Y)$ denote the set of edges with one end in $X$ and the other in $Y$, and  let $e(X,Y) = |E(X,Y)|$. Also, $\overline{e}(X,Y) = |X| \cdot |Y| - e(X,Y)$.

The main ingredient in the proof of Theorem \ref{main} is:

\begin{lemma}\label{mainlemma}
Let $G = (V,E)$ be an edge-coloured graph with the property that every colour class has size at most $k$ and there is no rainbow triangle.  If $n = |V|$ satisfies $n \ge n_1(k)$ (recall that $n_1(k)=6k^5(k+1)^2$) and  $m = |E|$ satisfies
\begin{equation}\label{massumption}
m \ge \frac{n^2}{4} + \frac{(k-1)n}{4} - \frac{k(k-1)}{2},\end{equation} and
\begin{equation}\label{deltaassumption}
\delta(G) \ge \lfloor \frac{n}{2} \rfloor,\end{equation}

\noindent then there exists a partition $\{X,Y\}$ of $V$ so that $X$ is an independent set and every component of $G[Y]$ has size at most $k$.
\end{lemma}

\begin{proof}

For $k=1$ the lemma follows from  Tur\'an's  Theorem, so we may assume that $k \ge 2$.  Let ${\mathcal T}$ be the set of triangles of $G$.  Now Lemma \ref{triangobs} together with a straightforward application of Cauchy-Schwarz yields
\begin{align*}
\sum_{uv \in E} |N(u) \cup N(v)|	
	&=	\sum_{uv \in E} \big( d(u) + d(v) - | \{ T \in {\mathcal T} \mid uv \in E(T) \}| \big)	\\
	&=	\sum_{v \in V} \left( d(v) \right)^2 - 3 | {\mathcal T} |	\\
	&\ge	\frac{4m^2}{n} - \frac{3}{2}m(k-1).
\end{align*}
Thus there exists an edge $x_0y_0$ such that $|N(x_0) \cup N(y_0)| \ge \frac{4m}{n} - \frac{3}{2}(k-1) \ge n - \frac{2}{3}(k-1)$, where for the second inequality we have used \eqref{massumption} and $n \ge n_1(k)$.  Choose disjoint sets $X_0 \subseteq N(x_0)$ and $Y_0 \subseteq N(y_0)$ so that $X_0 \cup Y_0 = N(x_0) \cup N(y_0)$.  Next we establish a claim which we will use to get upper bounds on $e(X_0)$ and $e(Y_0)$.

\bigskip

\noindent{\it Claim: } If $W \subseteq N(w)$ for some $w \in V$, then $e(W) \le (k-1)|W|$.

\smallskip

\noindent{\it Proof: } Let $I$ be the set of colours appearing on the edges $E(\{w\},W)$ and for every $i \in I$ let $d_i$ be the number of edges in $E(\{w\},W)$ of colour $i$. Let $F$ be the set of edges $uv\in E(W)$ for which $wu$ and $wv$ have the same colour. Then $|F| \le \sum_{i \in I}  \binom{d_i}{2}$. Since there are no rainbow triangles, every edge $uv \in E(W) \setminus F$ is coloured by some $i \in I$.
Since for every $i \in I$ the number of edges in $E(W)$ coloured $i$ is at most $k-d_i$, we have:
\begin{align*}
|E(W)|	&\le	\sum_{i \in I} \left( \binom{d_i}{2} + (k - d_i) \right)	\\
		&=	\sum_{i \in I} \big( (k-1)d_i  - (d_i - 1)( k - \tfrac{1}{2}d_i) \big)	\\
		&\le	(k-1) \sum_{i \in I} d_i	\\
		&= (k-1) |W|,
\end{align*}
establishing the claim.

Define $Z = V \setminus (X_0 \cup Y_0)$.  Applying the claim to each of $E(X_0)$ and $E(Y_0)$, we have:
\begin{equation*}
 m= |E(X_0)| + |E(Y_0)| + |E(X_0,Y_0)| +| E(Z)| +|E(Z,X_0 \cup Y_0)| \le (k-1) (|X_0| + |Y_0| ) + e(X_0,Y_0) + \sum_{z \in Z} d(z).
\end{equation*}
Since $|X_0| + |Y_0| \le n$ and $|Z| \le \frac{2}{3}(k-1)$, this implies
\begin{equation}
\label{xybd} m\le e(X_0,Y_0) + \frac{5}{3}(k-1)n.
\end{equation}

 Let $E(X,Y)$ be a maximum cut, meaning that $\{X,Y\}$ is a partition of $V$ for which $e(X,Y)$ is maximum.
  Let $a_X$ be the average degree of $G[X]$, let $a_Y$ be the average degree of $G[Y]$ and without loss of generality assume  that $a_X \le a_Y$.  We  proceed to prove a series of properties of $X$ and $Y$, eventually showing that they satisfy the conclusions of the lemma.

\begin{enumerate}[label=(\arabic*), labelindent=0em ,labelwidth=0cm, parsep=6pt, leftmargin =4mm]

\item $e(X,Y) \ge \frac{n^2}{4} - \frac{3}{2}(k-1)n$ and $\overline{e}(X,Y) \le \frac{3}{2}(k-1)n$.
\label{missinge}

The first part follows from  inequality \eqref{xybd},
 the fact that $e(X,Y) \ge e(X_0,Y_0)$  (since $(X,Y)$ is a maximum cut), inequality \eqref{massumption} and the assumption that $n \ge n_1(k)$.

The second part follows from the first since
\[ \overline{e}(X,Y) = |X| \cdot |Y| - e(X,Y) \le \tfrac{n^2}{4} - e(X,Y) \le \tfrac{3}{2}(k-1)n. \]

\item Every $x \in X$ satisfies $e(x,Y) \ge \frac{1}{2} \lfloor \frac{n}{2} \rfloor$ and every $y \in Y$ satisfies $e(y,X) \ge \frac{1}{2} \lfloor \frac{n}{2} \rfloor$.
\label{degreen4}

This follows from the assumption that the minimum degree is at least $\lfloor \frac{n}{2} \rfloor$, and the fact that a maximum cut is ``unfriendly'', namely every vertex is adjacent to at least as many vertices on the other side as on its own side.

\item Both  graphs $G[X]$ and $G[Y]$ have maximum degree less than $7k$.
\label{maxdegree}

Suppose for a contradiction that this is false, and assume (without loss of generality) that $x \in X$ has $d_X(x) \ge 7k$. Choose a set $X' \subseteq N(x) \cap X$ with $|X'| = 7k$ and let $I$ be the set of colours appearing on the edges $E(x, X')$.  Let $Y' = N(x) \cap Y$ and note that by (2) we must have $|Y'| \ge \frac{1}{2} \lfloor \frac{n}{2} \rfloor$.


Fix $y' \in Y'$. For each $x' \in N(y') \cap X'$, we have a triangle on $x,x',y'$. The majority colour in that triangle can be the colour of $xy'$ for at most $k-1$ choices of $x'$. For all other choices of $x'$, the colour of $x'y'$ must be in $I$. Hence $|N(y') \cap X'| \le k - 1 + | \{ e \in E(y',X') \mid \mbox{$e$ has a colour in $I$} \} |$. It follows that
\[ \overline{e}(y',X') \ge 6k + 1 -  | \{ e \in E(y',X') \mid \mbox{$e$ has a colour in $I$} \} |. \]
The total number of edges in $G$ with a colour in $I$ is at most $7k^2$.  Thus summing the above inequalities over all $y' \in Y'$ yields $\overline{e}(Y',X') \ge (6k+1)\frac{1}{2} \lfloor \frac{n}{2} \rfloor - 7k^2$ which exceeds $\frac{3}{2}(k-1)n$ for $n \ge n_1(k)$, contradicting \ref{missinge}.

\item $\lfloor \frac{n}{2} \rfloor - 7k < |X|, |Y| < \lceil \frac{n}{2} \rceil + 7k$.
\label{boundxysize}

To show, for example, that $|X| > \lfloor \frac{n}{2} \rfloor - 7k$, take any vertex $y \in Y$. By assumption \eqref{deltaassumption}, its degree in $G$ is at least $\lfloor \frac{n}{2} \rfloor$, and by the previous fact its degree in $G[Y]$ is less than $7k$. Hence $e(y,X) > \lfloor \frac{n}{2} \rfloor - 7k$, implying in particular  $|X| > \lfloor \frac{n}{2} \rfloor - 7k$. \\

Combining the last inequalities with \eqref{deltaassumption} yields:

\item Every $x \in X$ satisfies $\overline{e}(x,Y) < 14k$ and every $y \in Y$ satisfies $\overline{e}(y,X) < 14k$.
\label{missingdeg}

\item Every component of $G[X]$ and $G[Y]$ has size at most $k$.
\label{boundedcomp}

Suppose (for a contradiction) that there is a component of $G[X]$ with size greater than $k$ and choose a tree $T$ in $G[X]$ with exactly $k+1$ vertices.  It follows from \ref{missingdeg} that the number of common neighbours of $V(T)$ in the set $Y$ is at least $|Y| - 14k(k+1) \ge \frac{n}{2} - 14(k+1)^2 > k^2$, using $n \ge n_1(k)$.  Since there are at most $k^2$ edges which have the same colour as an edge of $T$, it follows that there is at least one vertex $y \in Y$ so that $y$ is adjacent to every vertex in $V(T)$ with the added property that no edge in $E(y, V(T))$ shares a colour with an edge in $E(T)$.  However, this gives a contradiction:  It is not possible for all edges in $E(y, V(T))$ to have the same colour (since this set has size $k+1$), but otherwise there is a rainbow triangle containing $y$ and an edge of $T$.

\item $a_X + a_Y \ge k - 1 - \frac{3k^2}{n}$.
\label{axaylower}

Note that $m \le \tfrac{1}{2} a_X |X| + \tfrac{1}{2} a_Y \left(n - |X| \right) + |X| \left( n - |X| \right)$.  The right hand side of this inequality is a quadratic function of $|X|$ for which the maximum is attained at
$|X| = \frac{n}{2} + \frac{a_X - a_Y}{4}$.  This maximum value is equal to $\frac{n^2}{4} + \frac{(a_X + a_Y)n}{4} + \frac{(a_X - a_Y)^2}{16}$.  It follows from \ref{boundedcomp} that $|a_X - a_Y| \le k-1$ and this gives us the following inequality which completes the argument.
\[ \tfrac{n^2}{4} + \tfrac{(k-1)n}{4} - \tfrac{k(k-1)}{2}  \le m \le \tfrac{n^2}{4} + \tfrac{(a_X + a_Y)n}{4} + \tfrac{(k-1)^2}{16}. \]

\item $a_X \le \frac{1}{k+1}$.
\label{minaxay}

(Remember that we assumed $a_X \le a_Y$ and our aim is to show that $a_X=0$.)
Suppose (for a contradiction) that $a_X > \frac{1}{k+1}$. Define $\tilde{a}_X = \lceil a_X - \frac{1}{k+1} \rceil$ and $\tilde{a}_Y = \lceil a_Y - \frac{1}{k+1} \rceil$ and note that our assumptions imply $\tilde{a}_X, \tilde{a}_Y \ge 1$.  Let $X^+ = \{ x \in X \mid d_X(x) \ge \tilde{a}_X \}$ and $X^- = X \setminus X^+$.  Then we have
\[ a_X |X| = \sum_{x \in X} d_X(x) \le (a_X - \tfrac{1}{k+1})|X^-| + k|X^+| = (a_X - \tfrac{1}{k+1}) |X| +  (k - a_X + \tfrac{1}{k+1}) |X^+|. \]
It follows that $|X^+| \ge \frac{|X|}{k(k+1)}$.  Similarly, setting $Y^+ = \{ y \in Y \mid d_Y(y) \ge \tilde{a}_Y \}$ we find that $|Y^+| \ge \frac{|Y|}{k(k+1)}$.  It follows from \ref{axaylower} that $\tilde{a}_X + \tilde{a}_Y \ge a_X + a_Y - \frac{2}{k+1} \ge k-1 - \frac{3k^2}{n} - \frac{2}{k+1} > k-2$, again using $n \ge n_1(k)$.  So, $\tilde{a}_X$ and $\tilde{a}_Y$ are positive integers which sum to at least $k-1$.

For every component of $G[X]$ of size at least $\tilde{a}_X+1$ choose a spanning tree, and let ${\mathcal T}_X$ be the set of these spanning trees.  Similarly let ${\mathcal T}_Y$ be a collection consisting of one spanning tree from each component of $G[Y]$ of size at least $\tilde{a}_Y+1$.  We claim that for every $T \in {\mathcal T}_X$ and $U \in {\mathcal T}_Y$ one of the following holds:
\begin{enumerate}
\item $\overline{e}( V(T), V(U) ) \ge 1$.
\item $E( V(T), V(U) )$ contains an edge with the same colour as an edge of $T$ or $U$.
\end{enumerate}
Were neither condition above to be satisfied, all edges between $V(T)$ and $V(U)$ would need to have the same colour to prevent a rainbow triangle, but then this colour would appear too many times. Indeed, $e(V(T),V(U)) \ge (\tilde{a}_X + 1)(\tilde{a}_Y +1)$ and this product exceeds $k$, because its two factors are greater than $1$ and sum to at least $k+1$.

For every $T \in {\mathcal T}_X \cup {\mathcal T}_Y$, the total number of edges in $G$ with the same colour as some colour in $E(T)$ is less than $k^2$.  Thus, by considering all pairs of trees from ${\mathcal T}_X \times {\mathcal T}_Y$ we get the following bound
\[ \overline{e}(X,Y) \ge | {\mathcal T}_X | \cdot | {\mathcal T}_Y | - k^2 | {\mathcal T}_X | - k^2 | {\mathcal T}_Y |. \]
Every vertex $x \in X^+$ has $d_X(x) \ge \tilde{a}_X$ and must therefore be in a component of $G[X]$ of size at least $\tilde{a}_X + 1$.  It follows that $| {\mathcal T}_X | \ge \frac{|X^+|}{k} \ge \frac{|X|}{k^2(k+1)}$.  A similar argument for $Y$ shows that
$| {\mathcal T}_Y | \ge \frac{|Y|}{k^2(k+1)}$.  Since (using $n \ge n_1(k)$) both of these quantities are larger than $k^2$ we can plug these bounds for $| {\mathcal T}_X|$ and $|{\mathcal T}_Y|$ in the previous equation.  Doing so gives the following inequality which gives us a contradiction (note that we have used \ref{boundxysize} to get the bound $|X| \cdot |Y| \ge \frac{n^2}{4} - 49k^2$ and, again, $n \ge n_1k)$).
\begin{align*}
\overline{e}(X,Y)
	&\ge \frac{ |X| \cdot |Y| }{ k^4(k+1)^2 } - \frac{|X| + |Y|}{k+1}  \\
    &\ge \frac{n^2}{4k^4(k+1)^2} - \frac{49}{k^2(k+1)^2} - \frac{n}{k+1} \\
    &\ge \left( \frac{3}{2}k - \frac{49}{k^2(k+1)^2n} - \frac{1}{k+1} \right)n	\\
    &> \tfrac{3}{2}(k-1)n.
\end{align*}

\end{enumerate}

With this last property in hand, we have all we need to complete the proof.  It follows from \ref{minaxay} and \ref{axaylower} and $n \ge n_1(k)$ that $a_Y \ge k - 1  - \frac{3k^2}{n} - \frac{1}{k+1} \ge k-1 - \frac{1}{k}$.  Note that the graph obtained from $K_k$ by removing one edge has average degree $k-1- \frac{2}{k}$;  more generally any graph on at most $k$ vertices which is not isomorphic to $K_k$ has average degree at most $k-1-\frac{2}{k}$.  Setting $t$ to be the number of components of $G[Y]$ isomorphic to $K_k$ we then have
\begin{align*}
\left( k - 1 - \tfrac{1}{k} \right) |Y|
	&\le		a_Y |Y| 	\\
	&\le		tk(k-1) + (|Y| - tk) (k - 1 - \tfrac{2}{k} ) \\
	&=		|Y|(k-1) - \tfrac{2}{k}( |Y| - tk ).
\end{align*}
It follows from this that $t \ge \frac{|Y|}{2k}$.  Let $s$ be the number of nontrivial components of $G[X]$ (i.e., components with at least two vertices).  By Lemma \ref{nokn}  there is at least one missing edge between every  $K_{k}$ subgraph of $G[Y]$ and every nontrivial component of $G[X]$.  Since each such component contains at most $\frac{k(k-1)}{2}$ edges, we have the bound
\[ m 	\le |Y| (n - |Y|) + \tfrac{k-1}{2}|Y| - s \left(t  - \tfrac{k(k-1)}{2} \right). \]
Now $|Y|(n-|Y|) + \tfrac{k-1}{2}|Y|$ is a quadratic function of $|Y|$ maximized at $|Y| = \frac{n}{2} + \frac{k-1}{4}$ with maximum value $\frac{n^2}{4} + \frac{(k-1)n}{4} + \frac{(k-1)^2}{16}$, so
\[ m \le \tfrac{n^2}{4} + \tfrac{ (k-1)n }{4} + \tfrac{ (k-1)^2}{16} - s \left(t  - \tfrac{k(k-1)}{2} \right). \]
Using our initial lower bound on $m$, and the facts that $t \ge \frac{|Y|}{2k}$, $|Y| > \lfloor \frac{n}{2} \rfloor -7k$, and $n \ge n_1(k)$, shows that $s > 0$ is impossible.  Therefore, $s=0$ and we have that $X$ is an independent set.  We established that all components of $G[Y]$ have size at most $k$ in \ref{boundedcomp} so the proof is now complete.\end{proof}

An easy calculation yields that removing a vertex of degree smaller than $\lfloor \frac{n}{2}\rfloor$ retains inequality \eqref{massumption}, namely:

\begin{lemma}
Let $G = (V,E)$ be a graph with   $n$ vertices and   $m $ edges, where
\begin{equation*}
m \ge \frac{n^2}{4} + \frac{(k-1)n}{4} - \frac{k(k-1)}{2}.\end{equation*}

\noindent Let $v$ be a vertex with degree smaller than $\lfloor \frac{n}{2}\rfloor$, let $m'=|E(G-v)|$ and let $n'=|V(G-v)|=n-1$. Then:
\begin{equation*}
m' \ge \frac{n'^2}{4} + \frac{(k-1)n'}{4} - \frac{k(k-1)}{2}.\end{equation*}

\end{lemma}

The next lemma shows that given inequality \eqref{massumption}, the process of removing vertices of degree smaller than half the current number of vertices cannot continue for too long.

\begin{lemma}\label{sqrt}
Let $G = (V,E)$ be a graph with   $n$ vertices and   $m $ edges, where $n \ge 2k$ and
\begin{equation}\label{massumption0}
m \ge \frac{n^2}{4} + \frac{(k-1)n}{4} - \frac{k(k-1)}{2}.\end{equation}


\noindent Then there exists a subset $U$ of $V(G)$ of size  $n' \ge \sqrt{n}$, such that the graph $G'=G[U]$ has minimal degree at least $\lfloor \frac{n'}{2}\rfloor$,  and its number of edges $m'$ satisfies
\begin{equation}\label{massumption2}
m' \ge \frac{n'^2}{4} + \frac{(k-1)n'}{4} - \frac{k(k-1)}{2}.\end{equation}

\noindent Moreover, $U$ can be obtained by removing sequentially vertices of degree smaller than half the current number of vertices.

\end{lemma}

\begin{proof}
If there is no vertex of degree smaller than $\lfloor \frac{n}{2} \rfloor$, there is nothing to prove.  Suppose that $G$ has a vertex, denote it $x_n$, with degree smaller than $\lfloor \frac{n}{2} \rfloor$. Let $G_{n-1}=G-v$. By the previous lemma, $G_{n-1}$ satisfies \eqref{massumption2}. If all vertices in $G_{n-1}$
have degree at least $\lfloor \frac{n-1}{2} \rfloor$, we are done. If not, remove a vertex $x_{n-1}$ of degree smaller than $\lfloor \frac{n-1}{2} \rfloor$. By the previous lemma the graph $G-x_n-x_{n-1}$ satisfies \eqref{massumption2}. We continue this way, until, having deleted
vertices $x_{n'+1}, x_{n'+2}, \ldots ,x_n$, we obtain a graph $G_{n'}$ all of whose vertices have degrees at least $\lfloor \frac{n'}{2} \rfloor$. The proof of the lemma will be complete if we show that $n'\ge \sqrt{n}$.

Let $U=\{x_1, \ldots,x_{n'}\}$  be the remaining set of vertices.
There are at most $\binom{n'}{2}$ edges of $G$ contained in $U$, and each vertex $x_i$, $i > n'$, is adjacent in $G$ to at most $\frac{i-1}{2}$ vertices $x_j, ~j<i$. Hence

$$m=|E(G)|\le \binom{n'}{2}+\sum_{i=n'+1}^n \frac{i-1}{2}= \binom{n'}{2} +\frac{1}{2}\left ( \binom{n}{2}-\binom{n'}{2} \right)=\frac{1}{2}\left ( \binom{n}{2}+\binom{n'}{2} \right)$$

This, combined with \eqref{massumption0} and the fact that $n \ge 2k$, yields $n'\ge \sqrt{n}$, completing the proof.\end{proof}

\bigskip

\noindent{\it Proof of Theorem \ref{main}.}
Let $k \ge 2$ and let $n \ge n_2(k)$. Let $G$ be an edge-coloured graph on $n$ vertices, with colour classes of size at most $k$ and no rainbow triangle, so that its number of edges $m$ is the largest among all such graphs. We shall show that there exists a graph $H^k_{a,b}$ with $a+b=n$ and $|E(H^k_{a,b})|\ge m$.

Clearly
\begin{equation*}
m = |E(G)| \ge |E(H^k_{ \lfloor \frac{n}{2} \rfloor, \lceil \frac{n}{2} \rceil})| \ge \tfrac{n^2}{4} + \tfrac{(k-1)n}{4} - \tfrac{k(k-1)}{2}.
\end{equation*}

Hence, by Lemma \ref{sqrt},
there exists a subset $U$ of $V=V(G)$  of size $n'\ge \sqrt{n}\ge n_1(k)$,
obtained by deleting sequentially vertices of degree less than half the current number of vertices,
such that   all degrees in $G[U]$ are at least $\lfloor \frac{n'}{2}\rfloor$
and \eqref{massumption2} is satisfied.
By Lemma \ref{mainlemma} applied to $G[U]$, there exists a partition $\{X',Y'\}$ of $U$ so that $X'$ is an independent set, and all components of $G[Y']$ have size at most $k$.

Starting from $G[U]$ and the partition $\{X',Y'\}$, we show how to construct a graph isomorphic to $H^k_{a,b}$, $a+b=n$, having at least $m$ edges. First, we put  vertices $z_{n'+1}, z_{n'+2}, \ldots ,z_n$ back into the graph, in this order, instead of the vertices  $x_{n'+1}, x_{n'+2}, \ldots ,x_n$ that were deleted. At each step we add $z_i$ to the smaller side  of the partition that we have at hand at that step (or to any side, if the two sides are equal), and we connect it to all vertices on the other side. By the property of $x_i$, the addition of $z_i$ adds more  edges
than the number of edges deleted when $x_i$ was removed. At the end of this process, we get a graph $\hat{G}$ with $n$ vertices and at least $m$ edges, together with a partition $\{X,Y\}$ of $V(\hat{G})$ satisfying the conclusion of Lemma \ref{mainlemma}. Note that at this stage we do not care about edge-colours, we just aim to ahow that the number of edges of $\hat{G}$ is at most that of $H^k_{|X|,|Y|}$.

If there is a component $H$ of $\hat{G}[Y]$ which is not a complete graph, modify $\hat{G}$ by adding all edges with both ends in $V(H)$ that are missing.  If there are two components of $\hat{G}[Y]$, say $H_1$ and $H_2$, both of size less than $k$, say $k > |V(H_1)| \ge |V(H_2)|$, increase $|E(\hat{G})|$ by deleting the edges $E(v, V(H_2) \setminus \{v\})$ for some $v \in V(H_2)$, and then addding the edges $E(v, V(H_1))$.  Finally, add to $\hat{G}$ any missing edge between $X$ and $Y$. Our modified $\hat{G}$ is now isomorphic to $H^k_{|X|,|Y|}$, and has at least $m$ edges. This proves the first part of the theorem. For the `moreover' part, observe that if $n' < n$ or if $\hat{G}$ is modified at all, we end up with strictly more than $m$ edges, contradicting the maximality of $m$.
\quad\quad$\Box$

\end{document}